\documentclass[a4]{article}
\usepackage{amsmath,  amssymb} 
\usepackage{amsthm}
\usepackage{amsfonts}
\usepackage{ascmac}

\usepackage[marginparwidth=0pt,margin=20truemm]{geometry} 

\usepackage{enumitem} 
\usepackage{graphicx}  
\usepackage{braket} 
\usepackage{bm} 
\usepackage{tikz-cd} 
\usepackage{graphicx}  
\usepackage{tikz-3dplot} 
\usepackage{comment} 

\usepackage[colorlinks=true,  linkcolor=blue,  citecolor=green]{hyperref}

\theoremstyle{definition}
\newtheorem{thm}{Theorem}[section]
\newtheorem{definition}[thm]{Definition}
\newtheorem{ex}[thm]{Example}
\newtheorem{lem}[thm]{Lemma}
\newtheorem{cor}[thm]{Corollary}
\newtheorem{prop}[thm]{Proposition}
\newtheorem{cond}[thm]{Condition}
\newtheorem{rem}[thm]{Remark}
\newtheorem{cons}[thm]{Construction}

\newtheorem{notation}[thm]{Notation}

\newcommand{\cone}{\mathring{c}}

\title{%
On $3$-dimensional locally standard $T$-pseudomanifolds
}

\author{
  Yuya Koike\thanks{Department of Pure and Applied Mathematics, Graduate School of Information Science and Technology, The University of Osaka, Suita, Osaka 565-0871, Japan. Email: \texttt{u175137d@ecs.osaka-u.ac.jp}} 
  }

\date{}

\begin{document}
\maketitle
\thispagestyle{empty}

\begin{abstract}
Locally standard $T$-pseudomanifolds were introduced by the authors in a previous work.
They are topological stratified pseudomanifolds equipped with torus actions. Their
equivariant homeomorphism types are classified by characteristic data under the homotopy
equivalence condition.
In this paper, we show that this condition can be removed when the
dimension is at most three. We also characterize those of dimension at most three that are
topological manifolds, in terms of their orbit spaces.
\end{abstract}

\tableofcontents

\section{Introduction}
In {\it toric topology}, the topology of manifolds is studied via torus actions, and the
correspondence between manifolds and their orbit spaces has been investigated. This field
has connections with several areas, including combinatorics, commutative algebra, algebraic
geometry, and symplectic geometry.
Davis and Januszkiewicz \cite{DJ91} introduced {\it quasitoric manifolds}. A quasitoric manifold
is a smooth $2n$-dimensional manifold with a locally standard $T^n$-action whose orbit
space is a simple polytope. They showed that the equivariant homeomorphism types of
quasitoric manifolds are completely classified by a characteristic pair consisting of the
orbit polytope and a characteristic function. Subsequently, toric topology has been
extended to broader classes of spaces and orbit spaces, and various generalizations of the
Davis-Januszkiewicz classification have been obtained (see, e.g., \cite{KK25}, \cite{MP06}, \cite{SS21}, \cite{Yos11}).
In a previous work \cite{KK26}, the authors introduced {\it locally standard
$T$-pseudomanifolds}.
This notion is a natural generalization of quasitoric manifolds,
extending the framework of toric topology from manifolds to {\it topological stratified
pseudomanifolds}.
Here, a topological stratified pseudomanifold is a generalization of a manifold
that allows cone-like singularities.
The main theorem of \cite{KK26} classifies the equivariant homeomorphism types of locally
standard $T$-pseudomanifolds satisfying the {\it homotopy equivalence condition}.
The classification is given in terms of {\it characteristic data} consisting of the orbit space, a
{\it characteristic functor}, and the {\it Chern class}.
This gives an extension of
the Davis-Januszkiewicz classification result.

In toric topology, orbit spaces have been primarily restricted to manifolds with corners or
simple convex polytopes.
This is because the theory has been developed mainly in the
framework of manifolds (and orbifolds).
In the framework of locally standard
$T$-pseudomanifolds, orbit spaces are topological stratified pseudomanifolds.
This class includes non-simple convex polytopes and cell decompositions of spheres. It was
also shown in \cite{KK26} that projective toric varieties, complete toric varieties, and
quasitoric manifolds are locally standard $T$-pseudomanifolds up to equivariant
homeomorphism.

However, \cite{KK26} left several problems open.
The classification theorem of \cite{KK26}
requires the \emph{homotopy equivalence condition}, namely that the ``interior" of the orbit
space is homotopy equivalent to the orbit space itself. 
Locally standard $T$-manifolds and complete toric varieties satisfy this condition.
A connected graph with two or more edges, however, does not.
Indeed, the interior of such a graph is
obtained by removing its vertices.
A classification without this assumption was not established in \cite{KK26}. Moreover, the geometry
and topology of locally standard $T$-pseudomanifolds were also left largely unexplored.
The definition of a locally standard $T$-pseudomanifold is inductive, and cone
singularities appear as open cones over lower-dimensional locally standard
$T$-pseudomanifolds. Therefore, the study of low-dimensional cases therefore serves as a natural
starting point for understanding the general structure.
This paper thus studies locally standard
$T$-pseudomanifolds of dimension at most three.

\subsection{Main results}

The main result of this paper is a classification of locally standard $T$-pseudomanifolds
of dimension at most three, without assuming the homotopy equivalence condition.
Every such space falls into one of the following three types, and we give a classification
for each type. The classification of each type is not restricted to dimension at most
three:

\begin{itemize}
  \item a principal $T$-bundle over a topological manifold;
  \item a {\it locally standard $T$-pseudomanifold over a graph}
        (Subsection~\ref{subsection over graph});
  \item a {\it locally standard $T$-pseudomanifold with a semi-trivial filtration}
        (Subsection~\ref{subsection semi-trivial}).
\end{itemize}
Furthermore, using the main classification theorem, we characterize locally standard
$T$-pseudomanifolds of dimension at most three that are topological manifolds, in terms of
their orbit spaces.
This characterization recovers a part of the classification of $3$-dimensional manifolds
with $S^1$-actions without exceptional orbits, due to Raymond \cite[Theorem~1]{Ray68}.
In this sense, the
main classification theorem of this paper extends a part of Raymond's result to the setting
of topological stratified pseudomanifolds.

\subsection{Structure of this paper}
In Section~\ref{sec 2}, we review the definition and the classification theorem of locally standard $T$-pseudomanifolds.
The material in this section is based on \cite{KK26}.
In Subsection~\ref{sec 2.1}, we recall the definition of topological stratified pseudomanifolds, and in Subsection~\ref{sec 2.2}, we introduce locally standard $T$-pseudomanifolds.
In Subsection~\ref{sec 2.3}, we present the main theorem of \cite{KK26}, namely the classification theorem for locally standard $T$-pseudomanifolds satisfying the homotopy equivalence condition.
This theorem asserts that equivariant homeomorphism types are classified by characteristic data.
In Subsection~\ref{sec 2.4}, we discuss canonical models, which play an important role in the proof of the classification theorem.
Canonical models also play an essential role in the present paper.

In Section~\ref{sec 3}, we classify equivariant homeomorphism types of locally standard $T$-pseudomanifolds of dimension at most three.
In these dimensions, the homotopy equivalence condition assumed in \cite{KK26} can be removed.
In Subsection~\ref{sec 3.1}, we classify characteristic functors, which form one of the ingredients of characteristic data.
In Subsections~\ref{subsection over graph} and~\ref{subsection semi-trivial}, we study locally standard $T$-pseudomanifolds over graphs and semi-trivial filtrations, respectively, both of which arise essentially in dimensions at most three.
In Subsection~\ref{sec 3.4}, we summarize the resulting classification results in tables.

In Section~\ref{sec 4}, we investigate when locally standard $T$-pseudomanifolds of dimension at most three admit structures of topological manifolds.
In Subsection~\ref{sec 4.1}, we characterize the manifold condition in terms of links of points.
Subsections~\ref{sec 4.2} and~\ref{sec 4.3} provide more explicit characterizations in several cases.
In Subsection~\ref{sec 4.4}, we summarize the classification results in tabular form.

Appendix~\ref{sec A} contains several standard lemmas on topological stratified pseudomanifolds.

\subsection{Acknowledgments}
The author is grateful to Nobutaka Asano for many fruitful discussions related to the research presented in this paper. The author would also like to thank Shintaro Kuroki for carefully reading an earlier version of the manuscript and for providing valuable comments and suggestions.

\section{Classification of locally standard $T$-pseudomanifolds}\label{sec 2}

In this section, we review locally standard $T$-pseudomanifolds introduced in \cite{KK26} and their classification theorem up to equivariant homeomorphism.
Locally standard $T$-pseudomanifolds are topological stratified pseudomanifolds equipped with torus actions.

\subsection{Topological stratified pseudomanifolds}\label{sec 2.1}

We begin by recalling the definition of a topological stratified pseudomanifold.
For details, see \cite[Chapter 2]{Fri20}.

A {\it filtered space} is a Hausdorff space $Q \neq \emptyset$ equipped with a filtration
\[
\mathfrak{Q} : Q = Q_{l+n} \supsetneq Q_{l+n-1} \supset Q_{l+n-2} \supset \cdots \supset Q_{l+i} \supset \cdots \supset Q_l \supsetneq Q_{l-1}=\emptyset, 
\quad (l \ge 0, \,  n \ge 0)
\]
by closed subsets.
We denote such a filtered space by $(Q, \mathfrak{Q})$ and use the notation
\[
\mathring{Q}:=Q_{l+n}\setminus Q_{l+n-1}.
\]
When $Q_i \setminus Q_{i-1} \neq \emptyset$,  each connected component of $Q_i \setminus Q_{i-1}$ is called a {\it stratum} of dimension $i$.
In particular,  a stratum of $\mathring{Q}$ is called a {\it top stratum}.
The space $Q_i$ is called the $i$-{\it skeleton}.
The index $l$ indicates the lowest dimension of the non-empty skeletons.
We call $n$ the {\it length} of $\mathfrak{Q}$.
If $Q_i \setminus Q_{i-1}=\emptyset$,  we omit $Q_i$ from the filtration.

To define a topological stratified pseudomanifold, we first introduce the notion of an {\it open cone}.
Open cones play an essential role in describing the local structure around singular points.
Here, a singular point means a point that does not admit a neighborhood homeomorphic to a Euclidean space.

\begin{definition}[open cone {{\cite[Definition 2.1.1]{Max19}}}]\label{def of open cone}
Let $L$ be a compact Hausdorff space. An {\it open cone on $L$} is defined by
\[
\mathring{c}(L):=L \times [0,  1) / (L \times \{0\}).
\] 
Namely,  $\mathring{c}(L)$ is the quotient space of $L\times [0, 1)$ obtained by collapsing the subspace $L\times \{0\}$ to a single point. We assume that \( \mathring{c}(\emptyset) \) is a single point. We call the point $[v, 0]\in \cone(L)$ a {\it cone vertex},  where $v$ is any element in $L$.
\end{definition}

We now define topological stratified pseudomanifolds.

\begin{definition}[topological stratified pseudomanifold {{\cite[Chapter 2]{Fri20}}}]
\label{def of pseudomanifold}

A {\it $0$-dimensional topological stratified pseudomanifold} is a set of points with the discrete topology.
For $n\ge 0$ and $l \ge 0$ with $l+n\neq 0$,  an {\it $(l+n)$-dimensional topological stratified pseudomanifold} is a filtered space $(Q, \mathfrak{Q})$ of length $n$ satisfying the following conditions:
\begin{enumerate}
\item
Every $(l+i)$-stratum is an $(l+i)$-dimensional topological manifold;

\item
The union $\mathring{Q}$ of top strata is dense in $Q$;

\item
For each point \( p \in Q_{l+i} \setminus Q_{l+i-1} \),  there exists a triple $(U_p,  L_p,  \varphi_p)$ consisting of
\begin{enumerate}
\item
an open neighborhood \( U_p \) of \( p \) in \( Q \);

\item
an $(n-i-1)$-dimensional compact topological stratified pseudomanifold \( L_p \),  possibly empty,  called a {\it link of $p$};

\item
a homeomorphism
\[
\varphi_p : U_p \to O_p \times \mathring{c}(L_p), 
\]
where \( O_p \) is a contractible open subset of \( \mathbb{R}^{l+i} \),  such that
\[
\varphi_p(U_p\cap Q_{l+i+j+1})
=
O_p\times \mathring{c}((L_p)_j)
\]
for every \( -1\le j\le n-i-1 \).
Here,  the filtration on $\mathring{c}(L_p)$ is given by
$
(\mathring{c}(L_p))_{j+1}
:=
\mathring{c}((L_p)_j).
$
\end{enumerate}
\end{enumerate}
\end{definition}

\begin{rem}[small open neighborhoods]\label{rem small nbh}
Let $Q$ be a topological stratified pseudomanifold and let $S\subset Q$ be a stratum.
An open neighborhood $U_p$ of a point $p\in S$ is called {\it small} if,  for every point $p'\in U_p$,  the stratum $S'$ satisfying $p'\in \overline{S'}$ also satisfies
$S\subset \overline{S'}$.
If $Q$ is compact,  then every point admits a small open neighborhood (see \cite[Appendix B]{KK26}).
Moreover,  small open neighborhoods form an open basis of $Q$.
\end{rem}

\subsection{Definition of a locally standard $T$-pseudomanifold}\label{sec 2.2}
We next introduce locally standard $T$-pseudomanifolds following \cite{KK26}.
Throughout this paper,  we use the following three notations:
\begin{itemize}
\item the circle group $U(1):=\{z\in\mathbb{C}\ |\ |z|=1\}$ that acts on $\mathbb{C}$ by scalar multiplication;
\item $\mathbb{C}^{\times}:=\mathbb{C}\setminus \{0\}$; 
\item $T^{m}\cong U(1)^{m}$ is an $m$-dimensional torus. We often denote it by $T$ when the dimension of $T$ is clear from the context.
\end{itemize}

Let $X$ be a second-countable,  compact Hausdorff space.
Suppose that $T^m$ acts continuously and effectively on $X$.
We assume that the isotropy subgroup $T_x$ is a subtorus
(i.e.,  a connected closed subgroup of $T^{m}$) for all $x\in X$.
For integers $l, n$ with $l\ge0$ and $m\ge n\ge0$, 
define the following subsets:
\begin{align*}
X_{l+2i+1+(m-n)}=X_{l+2i+(m-n)}:=\{x\in X\ |\ \dim T(x)\le i+(m-n)\}, 
\quad
(0\le i\le n), 
\end{align*}
where $T(x)$ is the $T^m$-orbit of $x\in X$.
The $T^m$-action induces the following {\it filtration by orbit dimension}:
\begin{align*}
\mathfrak{X}: X=X_{l+m+n}\supset X_{l+2(n-1)+(m-n)}\supset \cdots \supset X_{l+2i+(m-n)}\supset \cdots \supset X_{l+(m-n)}\supset \emptyset.
\end{align*}
Since $X_{l+2i+(m-n)}$ is $T$-invariant,  the orbit projection $\pi: X \to Q:=X/T$ induces the following filtration $\mathfrak{X}/T$ of the orbit space:
\begin{align*}
\mathfrak{X}/T: Q= Q_{l+n} \supset Q_{l+n-1} \supset \cdots \supset Q_{l+i}
\supset \cdots \supset Q_{l} \supset \emptyset, 
\end{align*}
where $Q_{l+i} := X_{l+2i+1+(m-n)}/T=X_{l+2i+(m-n)}/T $ for $0 \leq i \leq n$.
We call $(X/T,  \mathfrak{X}/T)$ the {\it filtered orbit space},  denoted by $(Q,  \mathfrak{Q})$.

In this paper,  we assume the following three conditions:
\begin{enumerate}[label=(C\arabic*),  ref=C\arabic*]
\item \label{cond-1}
$(X,  \mathfrak{X})$ is a topological stratified pseudomanifold (see Definition~\ref{def of pseudomanifold});
\item \label{cond-2}
$X_{l-2+(m-n)}=X_{l-1+(m-n)}:=\{x\in X \mid  \dim T(x)\le (m-n)-1 \} = \emptyset$,  i.e.,  $i=-1$;
\item \label{cond-3}
when $l+n\neq 0$,  we assume $X_{l+m+n} \supsetneq X_{l+m+n-2}$ (i.e.,  $X$ has free orbits).
\end{enumerate}

We now introduce the notion of a locally standard $T$-pseudomanifold.

\begin{definition}[{\cite[Definition 3.2]{KK26}}]\label{def of T-pseudomanifold}
An {\it $(l+m+n)$-dimensional locally standard $T$-pseudomanifold} $(X,  \mathfrak{X})$ is defined by induction on $(l+n)$ if it satisfies the following:
\begin{itemize}
\item for $n = 0$ and $l = 0$,  $X$ is a disjoint union of $m$-dimensional tori equipped with the multiplicative $T^{m}$-action (note that the compactness of $X$ implies that this is a finite disjoint union of $T^{m}$);

\item for $n\ge 0$ and $l \ge 0$,  with $l+n\neq 0$, the following condition holds:

 For any \( x \in X_{l+2i+(m-n)} \setminus X_{l+2(i-1)+(m-n)} \),  there exists a triple \( (U_x,  L_x,  \varphi_x) \) such that
\begin{enumerate}
  \item $U_x \subset X$ is a $T$-invariant open neighborhood of $x$;

  \item $L_x$ is a ($2n-2i-1$)-dimensional compact locally standard $T_x$-pseudomanifold,  possibly empty,  equipped with an action of $T_x \cong T^{n-i}$,  where $T_x$ is the isotropy subgroup of $x$ ($L_x$ is called a {\it link} of $x$);

 \item  \( \varphi_x : U_x \to O^l \times \big( \Omega \times U(1)^{m-n} \big) \times \mathring{c}({L}_x) \)
is a weakly equivariant homeomorphism, 
where $O^l \subset \mathbb{R}^l$ is a contractible open subset and $\Omega \subset (\mathbb{C}^{\times})^i$ is a $U(1)^i$-invariant open subset.
The torus
\( T^m \) acts on \(  O^l \times \big( \Omega \times U(1)^{m-n} \big) \times \mathring{c}({L}_x) \)
via an isomorphism
\[ T^m \cong T^m/{T_x} \times T_x \cong U(1)^{i+(m-n)} \times T^{n-i},  \]
where $U(1)^{i+(m-n)}$ acts on $\Omega \times U(1)^{m-n}$
by the standard multiplication (i.e.,  the free action),  and the $T_{x}$-action on $\mathring{c}(L_{x})=L_{x}\times [0, 1)/L_{x}\times \{0\}$ is induced by the $T_{x}$-action on the $L_{x}$-factor described in (2) and is trivial on the $[0, 1)$-factor.
\end{enumerate}
The open cone $\mathring{c}(L_x)$ inherits a natural filtration:
\begin{align*}
(\mathring{c}(L_{x}))_{j+1}:=\mathring{c}((L_{x})_{j}),  \quad (-1\le j\le 2n-2i-1).
\end{align*}
\end{itemize}
\end{definition}

We introduce notation for the free part of $X$ and the corresponding orbit space.

\begin{notation}\label{notation of free}
For a locally standard $T$-pseudomanifold $(X, \mathfrak{X})$
and a topological stratified pseudomanifold $(Q, \mathfrak{Q})$, 
we use the notation
\[
\mathring{X}
:=
X_{l+m+n}\setminus X_{l+m+n-2}, 
\qquad
\mathring{Q}
:=
Q_{l+n}\setminus Q_{l+n-1}.
\]
By definition,  $\mathring{X}$ consists of free orbits.
Hence,  the orbit projection
$
\pi|_{\mathring{X}}:\mathring{X}\to \mathring{Q}
$
is a principal $T$-bundle.
\end{notation}

\subsection{Classification theorem}\label{sec 2.3}

In this subsection, we review the classification theorem for locally standard $T$-pseudomanifolds established in \cite[Theorem 1.1]{KK26}.
A rough statement of the theorem is as follows
(the precise statement is given in Theorem~\ref{classification theorem intro}):
\[
\begin{tikzcd}[ampersand replacement=\&]
\left\{
\begin{array}{l}
\text{locally standard $T^m$-pseudomanifolds } X \\
\text{satisfying the homotopy equivalence condition}
\end{array}
\right\} \arrow[d,  leftrightarrow,  "{\Large 1:1}"] \\
\left\{
\text{orbit spaces } Q
\right\}
+
\left\{
\text{characteristic functors } \lambda
\right\}
+
\left\{
\text{Chern classes }
c \in H^2(Q; \mathbb{Z}^m)
\right\}
\end{tikzcd}
\]
The triple $(Q,  \lambda,  c)$ is called the {\it characteristic data}.
We now explain each component.

Let $X$ be an $(l+m+n)$-dimensional locally standard $T^m$-pseudomanifold.
The orbit space $Q=X/T^m$ has the structure of an $(l+n)$-dimensional topological stratified pseudomanifold (\cite[Proposition 4.1]{KK26}).
We next explain the characteristic functor.
The set of strata $\mathcal{S}(Q)$ of a topological stratified pseudomanifold has the structure of a poset category.
Namely, for strata $S_1,S_2 \in \mathcal{S}(Q)$, 
we define a morphism $S_1 \to S_2$ whenever
$S_1 \subset \overline{S_2}$,
where $\overline{S_2}$ denotes the closure of $S_2$.
Let $\mathcal{T}^m$ denote the category consisting of subtori of $T^m$,
whose morphisms are given by inclusions.
A {\it characteristic functor} is a functor
\begin{align}\label{def of char functor}
\lambda:\mathcal{S}(Q)^{\mathrm{op}}\to \mathcal{T}^m
\end{align}
such that, for each stratum $S$ of dimension $l+n-i$
(i.e., codimension $i$),
$\lambda(S)$ is an $i$-dimensional subtorus.
Here, $\mathcal{S}(Q)^{\mathrm{op}}$ denotes the opposite category of $\mathcal{S}(Q)$.
A characteristic functor can be constructed from a locally standard $T^m$-pseudomanifold $X$
(\cite[Section 5]{KK26}).
Indeed, the isotropy subgroup of a point $x\in X$
depends only on the stratum $S\subset Q$ containing $\pi(x)$.
Writing this subgroup as $T_S$,
we obtain a characteristic functor by setting
$\lambda(S)=T_S$.

We next explain the homotopy equivalence condition and the Chern class.

\begin{cond}[homotopy equivalence condition]\label{cond of homotopy}
The inclusion map $\iota: \mathring{Q} \to Q$ is a homotopy equivalence.
\end{cond}

Under this condition, the induced homomorphism
$\iota^{\ast}: H^2(Q;\mathbb{Z}^m) \to H^2(\mathring{Q}; \mathbb{Z}^m)$
is an isomorphism.
On the other hand, since the $T^m$-action on $\mathring{X}$ is free,
the orbit projection
$\pi|_{\mathring{X}}: \mathring{X} \to \mathring{Q}$
is a principal $T^m$-bundle.
Such principal $T^m$-bundles are classified by cohomology classes
$c^{\mathrm{free}} \in H^2(\mathring{Q}; \mathbb{Z}^m)$.
We call $c^{\mathrm{free}}$ the {\it Chern class of the principal bundle}.
Furthermore, since $\iota^{\ast}$ is an isomorphism,
we may define
$c:=(\iota^{\ast})^{-1}(c^{\mathrm{free}})
\in H^2(Q;\mathbb{Z}^m)$.
We call $c$ the {\it Chern class} of $X$.

We also recall the notion of (weak) isomorphisms between characteristic data
(\cite[Definition 6.5]{KK26}).

\begin{definition}\label{def of isomorphism of char pairs}
Let $f: Q \to Q'$ be a stratified homeomorphism (see \cite[Section 2]{Fri20}).
We say that $f$ is a {\it weak isomorphism between characteristic data}
$(Q,  \lambda,  c)$ and $(Q',  \lambda',  c')$
if the following conditions are satisfied:
\begin{enumerate}
    \item[(1)]
    $c=f^{\ast}(c')$,  where $f^{\ast} : H^2(Q'; \mathbb{Z}^m)
        \to
        H^2(Q; \mathbb{Z}^m)$
    is the induced isomorphism from $f$;

    \item[(2)]
    there exists an automorphism $\psi: T^m \to T^m$
    such that the following diagram commutes
    (i.e.,  for each stratum $S \in \mathcal{S}(Q)^{\mathrm{op}}$, 
    $\Psi \circ \lambda(S)
    =
    \lambda' \circ \mathcal{S}(f)(S)$):
    \begin{align}\label{comm diagram: weak isomorphism}
    \begin{tikzcd}[ampersand replacement=\&]
    \mathcal{S}(Q)^{\mathrm{op}}
        \ar{rr}{\mathcal{S}(f)}[']{\cong}
        \ar{d}[']{\lambda}
    \&
    \&
    \mathcal{S}(Q')^{\mathrm{op}}
        \ar{d}{\lambda'}
    \\
    \mathcal{T}
        \ar{rr}{\Psi}[']{\cong}
    \&
    \&
    \mathcal{T}
    \end{tikzcd}
    \end{align}
    where $\Psi$ is the isomorphism functor determined by
    $\Psi(T_S)=\psi(T_S)$ for $T_S \in \mathcal{T}$.
\end{enumerate}
If there exists such a weak isomorphism between
$(Q,  \lambda,  c)$ and $(Q', \lambda',  c')$, 
then we say that $(Q,  \lambda,  c)$ is
{\it weakly isomorphic} to $(Q', \lambda',  c')$.
When the automorphism $\psi$ is the identity, 
we simply say {\it isomorphism} instead of weak isomorphism, 
and {\it isomorphic} instead of weakly isomorphic.
\end{definition}

With these preparations, we can state the classification theorem.

\begin{thm}[classification theorem]\label{classification theorem intro}
Let $X$ and $X'$ be locally standard $T$-pseudomanifolds satisfying the homotopy equivalence condition (Condition~\ref{cond of homotopy}).
Then the following statements are equivalent:
\begin{enumerate}
    \item
    The characteristic data $(Q, \lambda, c)$ of $X$ and
    the characteristic data $(Q', \lambda', c')$ of $X'$
    are (weakly) isomorphic;

    \item
    $X$ and $X'$ are (weakly) $T$-equivariantly homeomorphic.
\end{enumerate}
\end{thm}

We next describe the characteristic data
and the notion of (weak) isomorphisms
for locally standard $T$-pseudomanifolds
without the homotopy equivalence condition.

\begin{rem}\label{rem free chern}
When a locally standard $T$-pseudomanifold $X$
does not necessarily satisfy the homotopy equivalence condition,
its characteristic data is defined to be
$
(Q,\lambda,c^{\mathrm{free}}).
$
Similarly, for two characteristic data
$
(Q,\lambda,c^{\mathrm{free}})
$
and
$
(Q',\lambda',(c^{\mathrm{free}})'),
$
one can define the notion of (weak) isomorphisms
by replacing condition {\rm(1)} in
Definition~\ref{def of isomorphism of char pairs}
with the following condition:
\begin{enumerate}
    \item[(1)']
    $
    c^{\mathrm{free}}
    =
    (f|_{\mathring Q})^{\ast}
    \bigl((c^{\mathrm{free}})'\bigr),
    $
    where
    $
    (f|_{\mathring Q})^{\ast}
    :
    H^2(\mathring Q'; \mathbb{Z}^m)
    \to
    H^2(\mathring Q; \mathbb{Z}^m)
   $
    is the induced isomorphism from the restriction
    $
    f|_{\mathring Q} : \mathring Q \to \mathring Q'.
   $
\end{enumerate}
Moreover, if locally standard $T$-pseudomanifolds
$X$ and $X'$ are (weakly) $T$-equivariantly homeomorphic,
then their characteristic data
$
(Q,\lambda,c^{\mathrm{free}})
$
and
$
(Q',\lambda',(c^{\mathrm{free}})')
$
are also (weakly) isomorphic in the above sense.
This can be shown in the same way as
\cite[Proposition 6.8]{KK26}.
In \cite{KK26}, this fact is first established
without assuming the homotopy equivalence condition,
and then the homotopy equivalence condition is used
to show that condition {\rm(1)} in
Definition~\ref{def of isomorphism of char pairs}
is equivalent to condition {\rm(1)'} above.
\end{rem}

\subsection{Canonical models}\label{sec 2.4}
A key point in the proof of the classification theorem
(Theorem~\ref{classification theorem intro}) is to show that a locally standard
$T$-pseudomanifold $X$ is equivariantly homeomorphic to the {\it canonical model}
$X(Q,\lambda,c)$ constructed from its characteristic data $(Q,\lambda,c)$.
Canonical models also play an essential role in this paper.

Let $Q$ be an $(l+n)$-dimensional topological stratified pseudomanifold, and let
$
  \lambda \colon \mathcal{S}(Q)^{\mathrm{op}} \to \mathcal{T}^m
$
be a characteristic functor on $Q$, where $m\ge n$.
Let
$
  c\in H^2(Q;\mathbb{Z}^m)
$
be a cohomology class.
We denote these data by $(Q,\lambda,c)$.
In other words, $(Q,\lambda,c)$ is an abstraction of the characteristic data of a locally standard $T$-pseudomanifold.
Weak isomorphisms between such data are defined in the same way as in Definition~\ref{def of isomorphism of char pairs}.
We now construct the canonical model $X(Q,\lambda,c)$ associated with $(Q,\lambda,c)$.

\begin{definition}\label{def of canonical model}
The {\it canonical model} $X(Q,  \lambda,  c)$ is defined as the quotient topological space
\[X(Q,  \lambda,  c):={P_c}/{\sim},  \]
where $\xi: P_c \to Q$ is a $T^m$-principal bundle over $Q$ whose Chern class is $c \in H^2(Q; \mathbb{Z}^m)$,  and the equivalence relation $\sim$ is defined as follows.
Two points $x$ and $y$ in $P_c$ are equivalent (denoted by $x\sim y$ or $x\sim_{\lambda} y$ if we emphasize the characteristic functor $\lambda$) if they satisfy $\xi(x)=\xi(y)=:p$ and $x$ and $y$ lie in the same $\lambda(S)$-orbit when $p \in {S}$ for some stratum $S$.
Moreover,  $X(Q,  \lambda,  c)$ has the canonical $T^m$-action induced by the action on $P_c$.
\end{definition}

\begin{rem}\label{face acyclic case}
When the second cohomology of $Q$ vanishes,  every $T^m$-principal bundle over $Q$ is trivial.
Therefore,  we have
\[
X(Q,  \lambda,  c)
=
X(Q,  \lambda,  0)
=
Q \times T^m / {\sim}
\]
where $(p, t) \sim (q, s)$ if $p=q$ and $t^{-1}s \in \lambda(S)$ whenever $p \in S$ for some stratum $S$.
\end{rem}

Canonical models satisfy the following property.

\begin{prop}[{\cite[Lemma 11.2]{KK26}}]\label{lem11.2}
Let $X$ be a locally standard $T$-pseudomanifold satisfying the homotopy equivalence
condition (see Condition~\ref{cond of homotopy}), and let $(Q, \lambda, c)$ be its
characteristic data. Then there exists an equivariant homeomorphism
\[
X \cong X(Q, \lambda, c).
\]
\end{prop}

\section{Classification of locally standard $T$-pseudomanifolds of dimension at most three}\label{sec 3}

In this section, let $X=X_{l+m+n}$ be a connected locally standard
$T$-pseudomanifold of dimension $l+m+n$ (see Definition~\ref{def of T-pseudomanifold}),
and let $\pi: X \to Q=Q_{l+n}$ be the orbit projection.
We use the notation
$\mathring{X}=X \setminus X_{l+m+n-2}$ and
$\mathring{Q}=Q \setminus Q_{l+n-1}$
(see Notation~\ref{notation of free}).
We assume that $m\ge 1$ and $\dim X=l+m+n \le 3$.
Then it follows that $\dim Q=l+n \le 2$.
The purpose of this chapter is to establish a classification theorem
for locally standard $T$-pseudomanifolds of dimension at most three.
Unlike \cite{KK26}, no homotopy equivalence condition is assumed.
From the definition of a locally standard $T$-pseudomanifold and the above assumptions,
the integers $(l,m,n)$ satisfy
\[
l \ge 0,  \quad
m \ge 1,  \quad
m \ge n \ge 0,  \quad
\dim X=l+m+n \le 3,  \quad
\dim Q=l+n \le 2.
\]
Hence, we obtain the following cases (Table~\ref{table1}).

\begin{table}[htbp]
\centering
\[
\begin{array}{c|c|ccc}
\dim X & \dim Q & l & m & n \\
\hline
1 & 0 & 0 & 1 & 0    \\
\hline
2 & 0 & 0 & 2 & 0   \\
2 & 1 & 1 & 1 & 0  \\
2 & 1 & 0 & 1 & 1 \\
\hline
3 & 0 & 0 & 3 & 0    \\
3 & 1 & 1 & 2 & 0 \\
3 & 1 & 0 & 2 & 1  \\
3 & 2 & 2 & 1 & 0  \\
3 & 2 & 1 & 1 & 1  \\
\end{array}
\]
\caption{Classification of the triples $(l, m, n)$}
\label{table1}
\end{table}

If $\dim Q=l+n=0$, then by the definition and the connectedness of $X$,
we have $X \cong T^m$, where the $T^m$-action is given by multiplication.
In particular, by Table~\ref{table1}, the following holds.

\begin{prop}\label{1-dim}
If $\dim X=1$, then $X \cong T^1$.
\end{prop}

When $n=0$, we have
$X=X_{l+m}=\mathring{X}$.
Therefore, $X$ is a principal $T^m$-bundle over $Q$.
Such bundles are classified by the homeomorphism type of $Q$ and a cohomology class $c \in H^2(Q; \mathbb{Z}^m)$.
Here, the base space $Q=Q_l$ is an $l$-dimensional topological manifold, and the characteristic functor $\lambda$ assigns the trivial subgroup $\{1\}$ to $Q$.

We next consider the case $n = 1$. In this case, $l$ takes one of two values: $0$ or $1$.
We introduce the following terminology.
\begin{itemize}
  \item If $n=1$ and $l = 0$, then $Q$ has the structure of a finite graph (possibly with
    loops and multiple edges) (see Lemma~\ref{structure of graph}). In this case, $X$
    is called a \emph{locally standard $T$-pseudomanifold over a graph}.
  \item If $n=1$ and $l = 1$, then $Q$ is called a \emph{$2$-stratifold}, and $X$ is
    called a \emph{locally standard $T$-pseudomanifold over a $2$-stratifold}.
\end{itemize}

\begin{rem}
We note that our notion of a $2$-stratifold differs slightly from the standard one.
By Definition~\ref{def of pseudomanifold}, $Q_1$ is a compact $1$-dimensional manifold.
Hence $Q_1 \cong S^1 \sqcup \cdots \sqcup S^1$
(see Lemma~\ref{classification of 1-manifold}).
In \cite{GGH16}, a $2$-stratifold is defined as a $2$-dimensional topological stratified pseudomanifold in which each $S^1$ component is contained in at least three $2$-strata.
This is because if an $S^1$ component is contained in only one $2$-stratum, the singularity is inessential, and if it is contained in exactly two $2$-strata, the space is locally homeomorphic to a connected sum of surfaces; thus, the singular points carry no essential geometric information in either case. 
Since \cite{GGH16} is primarily concerned with the homotopy types of $2$-stratifolds, such cases are excluded there.

On the other hand, in this paper, a characteristic functor assigns an isotropy subgroup to each stratum.
Therefore, even if a point is non-singular as a surface, the associated isotropy subgroup still carries geometric information.
We therefore include such cases in our framework and continue to use the term ``$2$-stratifold''.
\end{rem}

\subsection{Classification of characteristic functors}\label{sec 3.1}

For each orbit space $Q=Q_{l+n}$, characteristic functors
(see \eqref{def of char functor} in Subsection~\ref{sec 2.3}) can be classified easily.
When $n=0$, the space $Q=Q_l$ is a topological manifold, and the characteristic
functor assigns the trivial subgroup $\{1\}$ to $Q$ itself.
We next suppose that $n=1$.
In this case, the poset of strata $\mathcal{S}(Q)$ consists of
$(l+1)$-dimensional strata and $l$-dimensional strata.
The characteristic functor assigns the trivial subgroup $\{1\}$ to each
$(l+1)$-dimensional stratum, and assigns a $1$-dimensional subtorus of $T^m$
to each $l$-dimensional stratum.
In particular, when $m=1$, the only $1$-dimensional subtorus of $T^1$ is $T^1$
itself.
Hence the characteristic functor on $Q$ is uniquely determined.
We denote this characteristic functor by
$\lambda=\{1, T\}$.
Moreover, the stratum $Q_l$ corresponds to the fixed point set of $X$.

In Table~\ref{table1}, the remaining case is $(l,m,n)=(0,2,1)$.
This corresponds to locally standard $T$-pseudomanifolds over graphs.
In this case, the characteristic functor assigns a $1$-dimensional subtorus of
$T^2$ to each vertex of the graph, and the trivial subgroup $\{1\}$ to each edge.

\subsection{Locally standard $T$-pseudomanifolds over graphs}\label{subsection over graph}

In this subsection, let $X$ be a locally standard $T$-pseudomanifold over a graph $Q$.
Namely, we consider the case $l=0$ and $n=1$ in the notation $Q_{l+n}$.
The purpose of this subsection is to show that a locally standard
$T$-pseudomanifold over a graph is equivariantly homeomorphic to its canonical model (Proposition~\ref{canonical model of graph}):
\[
X \cong X(Q,  \lambda,  0)=Q \times T/{\sim_\lambda},
\]
and to establish the classification theorem for locally standard
$T$-pseudomanifolds over graphs
(Theorem~\ref{classification over graph}).
The proof proceeds by constructing a section
$s: Q \to X$
of the orbit projection
(Lemma~\ref{section of graph}).
Once such a section is constructed,
the theory of simple complexes of groups
(see \cite{BH99} and \cite{DLS19})
implies that $X$ admits a canonical model.
To prove the existence of a section,
we prepare the following two lemmas.

\begin{lem}\label{section of a loop}
  Let $Q$ be a graph consisting of a single loop, and let $X$ be a locally
  standard $T^m$-pseudomanifold over $Q$. Then there exists a section
  $s: Q \to X$.
\end{lem}
\begin{proof}
Let $v$ denote the unique vertex of the loop.
By examining the isotropy subgroups at each point, we see that a
$1$-dimensional subtorus $T_v \subset T^m$ occurs as the isotropy subgroup over $v$,
whereas the isotropy subgroup over every point of $Q\setminus \{v\}$ is the trivial subgroup $\{1\}$
(see \cite[Proposition 5.6]{KK26}).
Hence $T^m/T_v \cong T^{m-1}$ acts freely on $X$.
Let
\[
  T_{\mathrm{pin}} := X/(T^m/T_v)
\]
denote the orbit space of this free action, and let
\[
  \pi' \colon X \to T_{\mathrm{pin}}
\]
be the orbit projection.
The residual $T_v$-action on $T_{\mathrm{pin}}$ induces an orbit projection
\[
  \rho \colon T_{\mathrm{pin}} \to Q
\]
satisfying $\pi = \rho \circ \pi'$.
The situation may be summarized in the following commutative diagram:
\[
\begin{tikzcd}[ampersand replacement=\&]
 \&
 X  \ar{dl}{\pi'}[']{/(T/T_v)}   \ar{dr}{/T}[']{\pi}
 \&
 \\
   T_{\mathrm{pin}}   \ar{rr}{\rho}[']{/T_v}
   \&
   \&
    Q
\end{tikzcd}
\]
Let $\mathring{e}:=Q\setminus \{v\}$ denote the open edge of the loop.
Since $\mathring{e}$ is homeomorphic to $(0,1)$ and the isotropy subgroup over each point of $\mathring{e}$ is $\{1\}$, the restriction
$\rho^{-1}(\mathring{e}) \to \mathring{e}$
is a principal $T_v$-bundle over $(0,1)$.
Because $(0,1)$ is contractible, this bundle is trivial, and therefore
\[
  \rho^{-1}(\mathring{e}) \cong (0,1)\times T_v,
\]
which is homeomorphic to an open cylinder.
On the other hand, the fiber $\rho^{-1}(v)$ consists of a single point.
Consequently, $T_{\mathrm{pin}}$ is homeomorphic to the one-point compactification of an open cylinder, namely a pinched torus (see Figure~\ref{fig: pinched torus}).
\begin{figure}[htbp]
\centering
\begin{tikzpicture}[scale=0.5]
\fill[gray!20] (0,0) arc[start angle=90, end angle=450, x radius=3cm, y radius=3cm];
\fill[white] (0,0) arc[start angle=90, end angle=450, x radius=2cm, y radius=2cm];
\draw (0,0) arc[start angle=90, end angle=450, x radius=3cm, y radius=3cm];
\draw (0,0) arc[start angle=90, end angle=450, x radius=2cm, y radius=2cm];
\draw[dashed] (0,-4) arc[start angle=90, end angle=450, x radius=0.25cm, y radius=1cm];
\draw (0,-4) arc[start angle=90, end angle=270, x radius=0.25cm, y radius=1cm];
\fill (0,0) circle (2pt);
\node[above] at (0,0) {$\rho^{-1}(v)$};
\end{tikzpicture}
\caption{Pinched torus $T_{\mathrm{pin}}$}
\label{fig: pinched torus}
\end{figure}
It follows that $\rho \colon T_{\mathrm{pin}} \to Q$ admits a section
\[
  \overline{s} \colon Q \to T_{\mathrm{pin}}.
\]

We next consider the pull-back of the principal $T^{m-1}$-bundle $\pi' \colon X \to T_{\mathrm{pin}}$ along $\overline{s}$:
\[
  \overline{s}^{*}X
  =
  \{(u, x)\in Q \times X \mid \overline{s}(u)=\pi'(x)\}.
\]
This yields the pull-back diagram
\[
  \begin{tikzcd}[ampersand replacement=\&]
    \overline{s}^{\ast}X \ar{r}{\Phi}  \ar{d}{q}
    \&
    X   \ar{d}{\pi'}
    \\
    Q \ar{r}{\overline{s}}
    \&
    T_{\mathrm{pin}}
  \end{tikzcd}
\]
where $\Phi(u,x)=x$.
Since
$H^2(Q, \mathbb{Z}^{m-1}) = 0$,
the principal $T^{m-1}$-bundle
$q \colon \overline{s}^{*}X \to Q$
is trivial.
Hence it admits a section
\[
  \sigma \colon Q \to \overline{s}^{*}X.
\]
Define
\[
  s := \Phi \circ \sigma \colon Q \to X.
\]
We claim that $s$ is a section of $\pi$.
For each $u\in Q$, write
$
  \sigma(u)=(u,x_u)\in \overline{s}^{*}X.
$
By the definition of $\overline{s}^{*}X$, we have
$
  \overline{s}(u)=\pi'(x_u).
$
Since $s(u)=\Phi(\sigma(u))=x_u$, it follows that
$
  \pi'(s(u))
  =
  \pi'(x_u)
  =
  \overline{s}(u).
$
Thus
$
  \pi' \circ s = \overline{s}.
$
Combining this with $\rho \circ \overline{s}=\mathrm{id}_Q$ and
$\pi=\rho\circ\pi'$, we obtain
\[
  \pi \circ s
  =
  \rho \circ \pi' \circ s
  =
  \rho \circ \overline{s}
  =
  \mathrm{id}_Q.
\]
Therefore $s$ is a section of $\pi$.
\end{proof}

\begin{lem}\label{preimage of edge}
  Let $X$ be a locally standard $T$-pseudomanifold over a graph $Q$, and let
  $e \subset Q$ be an edge with two distinct vertices. Then $\pi^{-1}(e)$ is a
  locally standard $T$-pseudomanifold.
\end{lem}
\begin{proof}
It is straightforward to verify that $\pi^{-1}(e)$ is a topological stratified pseudomanifold.
Let $x \in \pi^{-1}(e)$.
Since $X$ is a locally standard $T$-pseudomanifold, there exists a chart
$
  (U_x, L_x, \varphi_x)
$
around $x$ in $X$.
We verify the conditions in Definition~\ref{def of T-pseudomanifold}.
We first suppose that
$
  x \in \pi^{-1}(e)\cap \mathring{X}
  =
  \pi^{-1}(\mathring{e}),
$
where $\mathring{e}$ denotes the relative interior of $e$.
Since the isotropy subgroup is trivial over $\mathring{e}$,
$
  \pi^{-1}(\mathring{e}) \to \mathring{e}
$
is a principal $T$-bundle.
As $\mathring{e}$ is contractible, this bundle is trivial.
Hence there exists a $T$-invariant open neighborhood
$
  U_x \subset \pi^{-1}(e)\cap \mathring{X}
$
of $x$.
In this case, we have $L_x=\emptyset$.

It remains to consider the case
$
  x \in \pi^{-1}(e)\cap X_{m-1}.
$
By Definition~\ref{def of T-pseudomanifold}, the link $L_x$ is a
$1$-dimensional locally standard $T^1$-pseudomanifold.
Hence
$
  L_x \cong T^1 \sqcup \cdots \sqcup T^1
$
by Proposition~\ref{1-dim}.
Moreover, the orbit space $L_x/T^1$ is a finite discrete set.
By \cite[Section 3]{Pop00} and \cite[Proposition 4.1]{KK26}, the following diagram is commutative, and a chart around $\pi(x)\in Q$ is given by
$
  (U_x/T, L_x/T_x, \varphi_x/T)
$
:
\[
\begin{tikzcd}[ampersand replacement=\&]
U_x   \ar{r}{\varphi_x}[']{\cong}
  \ar{d}{\pi|_{U_x}}
 \&
 U(1)^{m-1} \times \cone(T^1 \sqcup \cdots \sqcup T^1)
   \ar{d}{/U(1)^{m-1} \times T^1}
 \\
U_x/T      \ar{r}{\varphi_x/T}[']{\cong}
 \&
 \mathbb{R}^0 \times \cone (\ast \sqcup \cdots \sqcup \ast)
\end{tikzcd}
\]
Restricting $\varphi_x/T$ to the edge $e$, we obtain a homeomorphism
\[
(\varphi_x/T)|_{e \cap (U_x/T)}
:
e \cap (U_x/T)
\xrightarrow{\cong}
\mathbb{R}^0 \times \cone (\ast).
\]
Hence, by taking inverse images under the orbit projections and using the commutativity of the above diagram, we obtain a weakly $T$-equivariant homeomorphism
\[
\varphi_x |_{\pi^{-1}(e) \cap U_x}
:
\pi^{-1}(e) \cap U_x
\xrightarrow{\cong}
U(1)^{m-1} \times \cone (T^1).
\]
\end{proof}

We now construct a section by using the above lemmas.

\begin{lem}[section over a graph]\label{section of graph}
  Let $Q$ be a finite graph (possibly with loops and multiple edges), and let
  $X$ be a locally standard $T$-pseudomanifold over $Q$. Then there exists a
  section $s \colon Q \to X$.
\end{lem}
\begin{proof}
Let $\{v_1,\ldots,v_N\}$ be the vertex set of $Q$.
For each $1\le i\le N$, choose and fix a point
\[
  x_i \in \pi^{-1}(v_i).
\]
For each edge $e\subset Q$, let $v_i$ and $v_j$ denote its endpoints
(where $v_i=v_j$ if $e$ is a loop).
We shall construct a section
\[
  s_e \colon e \to \pi^{-1}(e)
\]
satisfying
\[
  s_e(v_i)=x_i,
  \quad
  s_e(v_j)=x_j.
\]
If $e$ is a loop, then such a section exists by Lemma~\ref{section of a loop}.
We next consider the case where $e$ is homeomorphic to a closed interval.
Identifying $e$ with $[0,1]$, we assume that $0$ and $1$ correspond to
$v_i$ and $v_j$, respectively.
By Lemma~\ref{preimage of edge}, the space $\pi^{-1}(e)$ is a locally standard
$T$-pseudomanifold over $e$.
Since $\pi^{-1}(e)$ satisfies the homotopy equivalence condition,
Proposition~\ref{lem11.2} yields a $T$-equivariant homeomorphism
\[
  \pi^{-1}(e)
  \cong
  e\times T^m/{\sim_{\lambda|_e}}
  \cong
  [0,1]\times T^m/{\sim_{\lambda|_e}},
\]
where $\lambda|_e$ denotes the restriction of the characteristic function
$\lambda$ to $e$.
Since $x_i,x_j\in \pi^{-1}(e)$ lie in the fibers over $v_i,v_j$,
respectively, they correspond to points
\[
  [0,t_{v_i}],
  \quad
  [1,t_{v_j}]
\]
for some $t_{v_i},t_{v_j}\in T^m$.
Because $T^m$ is path-connected, there exists a continuous path
$
  \gamma \colon [0,1]\to T^m
$
such that
\[
  \gamma(0)=t_{v_i},
  \quad
  \gamma(1)=t_{v_j}.
\]
Using $\gamma$, define a map
\[
  s'_e \colon [0,1]\to [0,1]\times T^m/{\sim}
\]
by
\[
  s'_e(u):=[u,\gamma(u)],
  \quad (u\in [0,1]).
\]
Then $s'_e$ is a section of
$
  [0,1]\times T^m/{\sim}\to [0,1].
$
Consequently, we obtain a section
\[
  s_e
  \colon
  e
  \xrightarrow{\cong}
  [0,1]
  \xrightarrow{s'_e}
  [0,1]\times T^m/{\sim}
  \xrightarrow{\cong}
  \pi^{-1}(e),
\]
which satisfies
\[
  s_e(v_i)=x_i,
  \quad
  s_e(v_j)=x_j.
\]

We next glue the sections $s_e$ over all edges $e\subset Q$.
Define
\[
  s(p):=s_e(p)
  \qquad (p\in e).
\]
Since each section $s_e$ was constructed so as to satisfy
$s_e(v_i)=x_i$ at every vertex $v_i$, the map $s$ is well-defined.
Moreover,
$
  \pi\circ s=\mathrm{id}_Q.
$
Because each edge is closed in $Q$, the gluing lemma implies that
$s$ is continuous.
Therefore $s$ is a section of $\pi$.
\end{proof}

We next show that a locally standard $T$-pseudomanifold over a graph admits a canonical model description.

\begin{prop}\label{canonical model of graph}
  For a locally standard $T$-pseudomanifold $X$ over a graph $Q$, the following
  equivariant homeomorphisms hold:
  \[
  X \cong Q \times T/{\sim_\lambda} \cong X(Q, \lambda, 0).
  \]
\end{prop}

\begin{proof}
By Lemma~\ref{section of graph}, the orbit projection $\pi:X\to Q$ admits a section.
Hence, by \cite[Proposition 12.20]{BH99}, we obtain a $T$-equivariant homeomorphism
$X \cong Q \times T/{\sim_\lambda}$.
Moreover, it follows from Remark~\ref{face acyclic case} that $Q \times T/{\sim_\lambda} \cong X(Q, \lambda, 0)$.
\end{proof}

We finally state the classification theorem for locally standard $T$-pseudomanifolds over graphs.

\begin{thm}\label{classification over graph}
  Let $X$ and $X'$ be locally standard $T$-pseudomanifolds
  over graphs, and let $(Q, \lambda, 0)$ and $(Q', \lambda', 0)$ be their
  respective characteristic data. Then the following two statements are
  equivalent:
  \begin{enumerate}
    \item $(Q, \lambda, 0)$ and $(Q', \lambda', 0)$ are (weakly) isomorphic (see Definition~\ref{def of isomorphism of char pairs} and Remark~\ref{rem free chern});
    \item $X$ and $X'$ are (weakly) equivariantly homeomorphic.
  \end{enumerate}
\end{thm}

\begin{proof}
Assume that $(Q,\lambda,0)$ and $(Q',\lambda',0)$ are (weakly) isomorphic.
Then, by \cite[Proposition 9.1]{KK26}, there exists a (weakly) $T$-equivariant homeomorphism
$X(Q,  \lambda,  0) \cong X(Q',  \lambda',  0)$.
Combining this with Proposition~\ref{canonical model of graph}, we obtain a (weakly) $T$-equivariant homeomorphism
$X \cong X'$.

The converse follows from Remark~\ref{rem free chern}.  
\end{proof}

\subsection{Locally standard $T$-pseudomanifolds with semi-trivial filtrations}\label{subsection semi-trivial}

When $l=m=n=1$, the filtration of $X$ takes the form
\[
X_3 \supset X_1 \supset \emptyset.
\]
Thus, $X$ consists only of free orbits and fixed points, which implies that the $T$-action is semi-free.
In this subsection, we study locally standard $T$-pseudomanifolds with such {\it semi-trivial filtrations} and establish a classification theorem for them.
We begin by generalizing this filtration structure to arbitrary dimensions.

\begin{definition}[semi-trivial filtration]\label{def of semi-trivial filtration}
  Suppose $m = n$, and let $(X, \mathfrak{X})$ be a $(l+2n)$-dimensional locally
  standard $T$-pseudomanifold. The filtration by orbit dimension $\mathfrak{X}$
  is called \emph{semi-trivial} if
  \[
  \mathfrak{X} \colon X_{l+2n} \supset X_l \supset \emptyset.
  \]
\end{definition}
\begin{rem}
  When $n=0$, the filtration of a locally standard $T$-pseudomanifold takes the form $X_{l+m} \supset \emptyset$. Such a filtration is called a \emph{trivial filtration} (see \cite{Fri20}), and in this case, the $T$-action is free.
  In Definition~\ref{def of semi-trivial filtration} above, the $T$-action is semi-free. Reflecting this property, we use the term \emph{semi-trivial}.
\end{rem}

Let $(X, \mathfrak{X})$ be a locally standard $T$-pseudomanifold with a semi-trivial filtration.
Then, the filtered orbit space $(Q, \mathfrak{Q})$ has the filtration $\mathfrak{Q}: Q_{l+n} \supset Q_l \supset \emptyset$.
The characteristic functor $\lambda$ on $Q$ is uniquely determined as the one assigning the torus $T^n$ to the $l$-strata and the trivial subgroup $\{1\}$ to the $(l+n)$-strata.
We denote this characteristic functor by $\lambda=\{1, T\}$.
Let $\mathring{X}:=X \setminus X_{l}$ and $\mathring{Q}:=Q \setminus Q_l$.
Restricting the orbit projection $\pi: X \to Q$ to the free orbits $\mathring{X}$ yields a principal $T^m$-bundle $\pi|_{\mathring{Q}}: \mathring{X} \to \mathring{Q}$.
We denote its Chern class by $c^{\mathrm{free}} \in H^2(\mathring{Q};\mathbb{Z}^m)$.
The triple $(Q, \{1, T\}, c^{\mathrm{free}})$ is called the \emph{characteristic data of the locally standard $T$-pseudomanifold with a semi-trivial filtration}.

To construct a canonical model from this data, we require the following condition on the cohomology of the orbit space.
\begin{cond}\label{cond surjective}
  The homomorphism
  \[
  \iota^{\ast} \colon H^2(Q; \mathbb{Z}^m) \to H^2(\mathring{Q}; \mathbb{Z}^m)
  \]
  induced by the inclusion map $\iota \colon \mathring{Q} \to Q$ is surjective.
\end{cond}

\begin{rem}
  This condition is satisfied when $H^2(\mathring{Q}; \mathbb{Z}^m) = 0$.
  More generally, the following argument holds whenever the preimage of
  $c^{\mathrm{free}} \in H^2(\mathring{Q}; \mathbb{Z}^m)$ under $\iota^{\ast}$
  exists, but we do not pursue this generality here.
\end{rem}

Under this assumption, we construct a canonical model from characteristic data.

\begin{cons}\label{cons of canonical model}
  Let $X$ be a locally standard $T$-pseudomanifold with a semi-trivial filtration,
  and let $(Q, \{1,T\}, c^{\mathrm{free}})$ be its characteristic data.
By the surjectivity of $\iota^{\ast}$, we may take
  \[
  c \in (\iota^{\ast})^{-1}(c^{\mathrm{free}}) \subset H^2(Q; \mathbb{Z}^m).
  \]
Let $P_c$ denote the principal $T^m$-bundle over $Q$ whose Chern class is $c$.
The restriction of $P_c$ to $\mathring{Q}$ is bundle isomorphic to $\mathring{X}$.
We define the canonical model by
  \[
  X(Q, \{1,T\}, c^{\mathrm{free}}) := X(Q, \{1,T\}, c) = P_c / {\sim_{\{1,T\}}},
  \]
  where $\sim_{\{1,T\}}$ is the equivalence relation collapsing the fibers over
  $Q_l$ (see also Definition~\ref{def of canonical model} for the canonical model $X(Q, \{1,T\}, c)$). By \cite[Theorem 12.1]{KK26}, $X(Q, \{1,T\}, c)$ is a locally standard
  $T$-pseudomanifold. The filtration of $X(Q, \{1,T\}, c^{\mathrm{free}}) =
  X(Q, \{1,T\}, c)$ is given by
  \[
  X(Q, \{1,T\}, c) \supset X(Q_l, \{1,T\}, c) \supset \emptyset,
  \]
  where $X(Q_l, \{1,T\}, c) \cong Q_l$. We also set
  $X(\mathring{Q}, \{1,T\}, c) := X(Q, \{1,T\}, c) \setminus X(Q_l, \{1,T\}, c)$.
\end{cons}

\begin{rem}
As will be shown later in Proposition~\ref{key lem}, we have
$X \cong X(Q, \{1, T\}, c^{\mathrm{free}})$.
Therefore, the equivariant homeomorphism type of
$X(Q, \{1,T\}, c^{\mathrm{free}})$ does not depend on the choice of
$c \in (\iota^{\ast})^{-1}(c^{\mathrm{free}})$.
\end{rem}

To show $X \cong X(Q, \{1,T\}, c^{\mathrm{free}})$,
we construct the following map.

\begin{cons}\label{cons of h}
  Let $X$ be a locally standard $T$-pseudomanifold with a semi-trivial filtration,
  let $(Q, \{1,T\}, c^{\mathrm{free}})$ be its characteristic data, and let
  $X(Q, \{1,T\}, c) = X(Q, \{1,T\}, c^{\mathrm{free}})$ be its canonical model.
  Let $\pi \colon X \to Q$ and $\tau \colon X(Q, \{1,T\}, c) \to Q$ be the
  respective orbit projections. We construct a map
  \[
  h \colon X \to X(Q, \{1,T\}, c)
  \]
  as follows. Since $\mathring{X} \cong P_c |_{\mathring{Q}} \cong X(\mathring{Q}, \{1,T\}, c)$, we may take
  an equivariant homeomorphism
  $\mathring{h} \colon \mathring{X} \to X(\mathring{Q}, \{1,T\}, c)$.
  Since $X_l$ and $X(Q_l, \{1,T\}, c)$ consist of fixed points, the restrictions
  of $\pi$ and $\tau$ to $Q_l$ are homeomorphisms, and we obtain a homeomorphism
  $\tau^{-1}|_{Q_l} \circ \pi|_{Q_l} \colon X_l \to X(Q_l, \{1,T\}, c)$.
  We define $h$ by
  \[
  h(x) :=
  \begin{cases}
    \mathring{h}(x), & (x \in \mathring{X}); \\
    \tau^{-1}|_{Q_l} \circ \pi|_{Q_l}(x), & (x \in X_l).
  \end{cases}
  \]
  By construction, $h$ is a bijection, and the following diagram commutes:
  \[
  \begin{tikzcd}[ampersand replacement=\&]
    X \ar{rr}{h} \ar{rd}{\pi}
    \&
    \&
    X(Q, \{1,T\}, c) \ar{ld}{\tau}
    \\
    \&
    Q
    \&
  \end{tikzcd}
  \]
  In particular,
  \begin{align}\label{correspond of invariant nbhd}
    h(\pi^{-1}(U)) = \tau^{-1}(U)
  \end{align}
  holds for any subset $U \subset Q$.
\end{cons}

To prove the continuity of $h$, we establish the following lemma.

\begin{lem}\label{lem of inv nbh}
  For any fixed point $x \in X(Q_l, \{1,T\}, c)$ and any open neighborhood
  $V_x$ of $x$, there exists a $T$-invariant open neighborhood $U_x$ of $x$
  contained in $V_x$.
\end{lem}
\begin{proof}
Consider the torus action $\varphi:T \times X(Q,  \{1,T\},  c) \to X(Q,  \{1,T\},  c)$.
Since $x$ is a fixed point, we have 
  \[
  T \times \{ x \} \subset \varphi^{-1}(V_x).
  \]
The subset $\varphi^{-1}(V_x)$ is open in $T \times X(Q,\{1,T\},c)$, and $T$ is compact.
Hence, by the tube lemma, there exists an open neighborhood $W_x$ of $x$
in $X(Q,\{1,T\},c)$ such that
  \[
  T \times \{x\} \subset T \times W_x \subset \varphi^{-1}(V_x).
  \]
We define $U_x:=\varphi( T \times W_x)$.
Then $U_x \subset V_x$, and $U_x$ is $T$-invariant by construction.
It remains to verify that $U_x$ is open.
Since
  \[
  U_x=\varphi( T \times W_x)=T \cdot W_x=
  \bigcup_{t \in T}t \cdot W_x,
  \]
it suffices to show that each subset $t\cdot W_x$ is open.
For each $t\in T$, the map $\varphi_t:X(Q,  \{1,T\},  c) \to X(Q,  \{1,T\},  c)$ is a homeomorphism.
Therefore $\varphi_{t}(W_x)=t\cdot W_x$ is open.
\end{proof}

Using this lemma, we show that $h$ is continuous.

\begin{lem}\label{h is cont}
  The map $h \colon X \to X(Q, \{1,T\}, c)$ is continuous.
\end{lem}
\begin{proof}
To prove the continuity of $h$, it suffices to show that for every point
$x\in X(Q,\{1,T\},c)$ and every small open neighborhood $V_x$ of $x$
(see Remark~\ref{rem small nbh}), the preimage
$h^{-1}(V_x)$
is open in $X$.
We first suppose that
$x \in X(\mathring{Q},  \{1,T\},  c)$.
Since $V_x$ is small, we have
$V_x \subset X(\mathring{Q},  \{1,T\},  c)$.
By the definition of $h$, we obtain
 $h^{-1}(V_x) = \mathring{h}^{-1}(V_x)$.
  Since
$\mathring{h}^{-1}:   X(\mathring{Q},\{1,T\},c) \to
\mathring{X}$
 is a homeomorphism,
 the subset
  $\mathring{h}^{-1}(V_x)=h^{-1}(V_x)$
is open in $\mathring{X}$.
Because $\mathring{X}$ is open in $X$, it follows that
$h^{-1}(V_x)$
is open in $X$.

We next suppose that
$x \in X(Q_l,  \{1,T\},  c)$.
Then $x$ is a fixed point.
By Lemma~\ref{lem of inv nbh}, the neighborhood $V_x$ contains a
$T$-invariant open neighborhood $U_x$ of $x$.
Hence we may write
\[
  V_x = U_x \cup W,
\]
where $W$ is an open subset contained in
$
  X(\mathring{Q},\{1,T\},c).
$
By the previous argument,
$
  h^{-1}(W)
  =
  \mathring{h}^{-1}(W)
$
is open in $X$.
On the other hand, since $U_x$ is $T$-invariant, it can be written as
\[
  U_x=\tau^{-1}(U_p),
\]
where
$
  p:=\tau(x),
$
and $U_p$ is a small open neighborhood of $p$ in $Q$.
By \eqref{correspond of invariant nbhd},
\[
  h^{-1}(\tau^{-1}(U_p))
  =
  \pi^{-1}(U_p).
\]
Since $U_p$ is open in $Q$, the subset
$
  \pi^{-1}(U_p)
$
is open in $X$.
Therefore
\[
  h^{-1}(U_x)
\]
is open in $X$.
Consequently,
\[
  h^{-1}(V_x)
  =
  h^{-1}(U_x)\cup h^{-1}(W)
\]
is open in $X$.
Hence $h$ is continuous.
\end{proof}

We are now ready to show that the map $h$ is an equivariant homeomorphism.
\begin{prop}\label{key lem}
  Let $X$ be a locally standard $T$-pseudomanifold with a semi-trivial filtration
  satisfying Condition~\ref{cond surjective}, and let $(Q, \{1,T\}, c^{\mathrm{free}})$
  be its characteristic data. Then the following equivariant homeomorphism holds:
  \[
  X \cong X(Q, \{1,T\}, c^{\mathrm{free}}).
  \]
\end{prop}
\begin{proof}
By Lemma~\ref{h is cont}, the bijective map $h$ is continuous.
Since $X$ is compact and $X(Q, \{1,T\}, c^{\mathrm{free}})$ is Hausdorff, it follows that $h$ is a homeomorphism.
It remains to show that $h$ is equivariant. By the definition of $h$, it suffices to verify that both $\mathring{h}$ and
$
\tau^{-1}|_{Q_l} \circ \pi|_{Q_l}
$
are equivariant.
The bundle isomorphism $\mathring{h}$ is equivariant by definition. Moreover, the map
\[
\tau^{-1}|_{Q_l} \circ \pi|_{Q_l} \colon X_l \to X(Q_l, \{1,T\}, c^{\mathrm{free}})
\]
is also equivariant, since both $X_l$ and $X(Q_l, \{1,T\}, c^{\mathrm{free}})$ consist of fixed points.
\end{proof}

As a corollary of this proposition, we obtain the case $(l, m, n)=(1, 1, 1)$.

\begin{cor}\label{canonical model of 111}
  Let $X$ be a $3$-dimensional locally standard $T$-pseudomanifold over a
  $2$-stratifold. Then $c^{\mathrm{free}} = 0$, and the following equivariant
  homeomorphisms hold:
  \[
  X \cong X(Q, \{1,T\}, 0) \cong Q \times T / {\sim_{\{1,T\}}}.
  \]
\end{cor}

\begin{proof}
By Lemma~\ref{properties of top stratum}-(2)(b), 
$\mathring{Q}$ is a non-compact $2$-dimensional manifold.
Hence, \cite[Proposition 3.29]{Hat02} implies $H^2(\mathring{Q};\mathbb{Z}^m)=0$, which in particular yields $c^{\mathrm{free}}=0$.
Therefore, the homomorphism
$\iota^{\ast}: H^2(Q ;\mathbb{Z}^m) \to  H^2(\mathring{Q} ;\mathbb{Z}^m)=0$
induced by the inclusion
$\iota: \mathring{Q} \to Q$
is surjective.
The assertion then follows from Proposition~\ref{key lem}.
\end{proof}

Using the canonical models, we establish the main classification theorem for spaces with semi-trivial filtrations.

\begin{thm}\label{classification of semi-trivial}
  Let $X$ and $X'$ be locally standard $T$-pseudomanifolds with semi-trivial
  filtrations satisfying Condition~\ref{cond surjective}, and let
  $(Q, \{1,T\}, c^{\mathrm{free}})$ and $(Q', \{1,T\}, (c^{\mathrm{free}})')$
  be their respective characteristic data. Then the following two statements
  are equivalent:
  \begin{enumerate}
    \item $(Q, \{1,T\}, c^{\mathrm{free}})$ and $(Q', \{1,T\}, (c^{\mathrm{free}})')$
      are (weakly) isomorphic;
    \item $X$ and $X'$ are (weakly) equivariantly homeomorphic.
  \end{enumerate}
\end{thm}

\begin{proof}
Assume that
$
(Q, \{1,T\}, c^{\mathrm{free}})
$
and
$
(Q', \{1,T\}, (c^{\mathrm{free}})')
$
are (weakly) isomorphic.
By definition, there exists a stratified homeomorphism
$
f \colon Q \to Q'.
$
Since a stratified homeomorphism between topological stratified pseudomanifolds preserves the filtration, the restriction
\[
f|_{\mathring{Q}} \colon \mathring{Q} \to \mathring{Q'}
\]
is also a homeomorphism.
Then, the following diagram commutes:
\[
\begin{tikzcd}[ampersand replacement=\&]
Q  \ar{r}{f}[']{\cong}
 \&
   Q'
   \\
\mathring{Q}  \ar{r}{f|_{\mathring{Q}}}[']{\cong}  \ar{u}{\iota}
 \&
  \mathring{Q'}  \ar{u}{\iota'}
\end{tikzcd}
\]
where $\iota$ and $\iota'$ denote the inclusion maps.
This commutative diagram induces the following commutative diagram in cohomology:
\[
\begin{tikzcd}[ampersand replacement=\&]
H^2(Q; \mathbb{Z}^m) \ar{d}{\iota^{\ast}}
 \&
   H^2(Q'; \mathbb{Z}^m)
   \ar{l}{f^{\ast}}[']{\cong}
   \ar{d}{(\iota')^{\ast}}
   \\
H^2(\mathring{Q}; \mathbb{Z}^m)
 \&
  H^2(\mathring{Q'}; \mathbb{Z}^m)
  \ar{l}{(f|_{\mathring{Q}})^{\ast}}[']{\cong}
\end{tikzcd}
\]
Since
$
(Q, \{1,T\}, c^{\mathrm{free}})
$
and
$
(Q', \{1,T\}, (c^{\mathrm{free}})')
$
are (weakly) isomorphic, Remark~\ref{rem free chern}-(1)' implies that
\begin{align}\label{relation of chern class}
(f|_{\mathring{Q}})^{\ast}((c^{\mathrm{free}})')
=
c^{\mathrm{free}}.
\end{align}
Take a cohomology class
\[
c'
\in
(\iota'^{\ast})^{-1}((c^{\mathrm{free}})')
\subset
H^2(Q'; \mathbb{Z}^m),
\]
and define
\begin{align}\label{take c}
c := f^{\ast}(c')
\in
H^2(Q; \mathbb{Z}^m).
\end{align}
Then, by \eqref{relation of chern class} and the commutativity of the above diagram, we obtain
\[
\iota^{\ast}(c)=c^{\mathrm{free}}.
\]
Hence, by Construction~\ref{cons of canonical model},
we have
\[
X(Q, \{1,T\}, c^{\mathrm{free}})
=
X(Q, \{1,T\}, c),
\quad
X(Q', \{1,T\}, (c^{\mathrm{free}})')
=
X(Q', \{1,T\}, c').
\]
Furthermore, Proposition~\ref{key lem} yields equivariant homeomorphisms
\begin{align}\label{by key lem}
X \cong X(Q, \{1,T\}, c),
\quad
X' \cong X(Q', \{1,T\}, c').
\end{align}
Moreover, by \eqref{take c},
$
(Q, \{1,T\}, c)
$
and
$
(Q', \{1,T\}, c')
$
are (weakly) isomorphic (see Definition~\ref{def of isomorphism of char pairs}-(1)).
Therefore, by \cite[Theorem 9.1]{KK26}, there exists a (weakly) equivariant homeomorphism
\begin{align}\label{kk26 thm 9.1}
X(Q, \{1,T\}, c)
\cong
X(Q', \{1,T\}, c').
\end{align}
Combining \eqref{by key lem} and \eqref{kk26 thm 9.1}, we obtain a
(weakly) equivariant homeomorphism
\[
X \cong X'.
\]

The converse follows from Remark~\ref{rem free chern}.
\end{proof}

As an immediate consequence of Theorem~\ref{classification of semi-trivial}, we obtain the following classification for the case of $2$-stratifolds.

\begin{cor}\label{class of over 2-stratifold}
  Let $X$ and $X'$ be $3$-dimensional locally standard $T$-pseudomanifolds
  over $2$-stratifolds (i.e., $(l,m,n) = (1,1,1)$), and let $(Q, \{1,T\}, 0)$
  and $(Q', \{1,T\}, 0)$ be their respective characteristic data (see Corollary~\ref{canonical model of 111}). 
  Then the following two statements are
  equivalent:
  \begin{enumerate}
    \item $(Q, \{1,T\}, 0)$ and $(Q', \{1,T\}, 0)$ are (weakly) isomorphic;
    \item $X$ and $X'$ are (weakly) equivariantly homeomorphic.
  \end{enumerate}
\end{cor}

\subsection{Classification table}\label{sec 3.4}

We conclude this section by summarizing the above results in the following table.
Here $\{\ast\}$ denotes a one-point set, $\{1\}$ denotes the characteristic
functor assigning the trivial subgroup $\{1\}$ to itself.

\begin{table}[h]
\centering
\[
\begin{array}{c|cccc|c}
\dim X & \dim Q & l & m & n & X\\
\hline
1 & 0 & 0 & 1 & 0 & T^1 \\
\hline
2 & 0 & 0 & 2 & 0 & T^2  \\
2 & 1 & 1 & 1 & 0 & \text{principal $T^1$-bundle} \cong S^1 \times T^1\\
2 & 1 & 0 & 1 & 1 & \text{ over graph}\\
\hline
3 & 0 & 0 & 3 & 0 &  T^3  \\
3 & 1 & 1 & 2 & 0 & \text{principal $T^2$-bundle} \cong S^1 \times T^2\\
3 & 1 & 0 & 2 & 1 & \text{over graph}\\
3 & 2 & 2 & 1 & 0 & \text{principal $T^1$-bundle}\\
3 & 2 & 1 & 1 & 1 & \text{over $2$-stratifold}\\
\end{array}
\qquad
\begin{array}{c}
 \text{characteristic data } (Q, \lambda, c)\\
\hline
 (\{ \ast \}, \{1\}, 0) \\
\hline
(\{ \ast \}, \{1\}, 0) \\
 (S^1, \{1\}, 0)\\
 (\text{graph}, \{1,T\}, 0)\\
\hline
 (\{ \ast \}, \{1\}, 0)  \\
 (S^1, \{1\}, c)\\
 (\text{graph}, \lambda, 0)\\
 (\text{closed surface}, \{1\}, c)\\
 (\text{$2$-stratifold}, \{1,T\}, 0) \\
\end{array}
\]
\caption{Classification of locally standard $T$-pseudomanifolds of dimension at most $3$ and their characteristic data.}
\label{tab:classification}
\end{table}

\section{Manifold structures on locally standard $T$-pseudomanifolds}\label{sec 4}

Let $X$ be a connected locally standard $T$-pseudomanifold with
$\dim X \leq 3$.
In this section, we determine when $X$ admits a structure of a topological manifold in each case appearing in the classification table (see Table~\ref{tab:classification}).
When $n=0$, the orbit space $Q=Q_l$ is a topological manifold.
Moreover, since $X$ is a principal $T^m$-bundle over $Q$, it follows that $X$ is also a topological manifold.
In this case, $Q_0$ is a point, $Q_1$ is homeomorphic to $S^1$, and $Q_2$ is a closed surface.

In the following, we consider the case $n=1$.
Namely, $X$ is a locally standard $T$-pseudomanifold over either a graph or a $2$-stratifold.

\subsection{Links}\label{sec 4.1}
In this subsection, we study links of points in low-dimensional locally standard
$T$-pseudomanifolds and relate them to manifold structures and normality.
We first determine the possible links in low-dimensional
locally standard $T$-pseudomanifolds.
\begin{lem}\label{link is simple}
Let $X$ be a locally standard $T$-pseudomanifold with $\dim X \leq 3$ and $n = 1$.
  Then for each point $x \in X$, the link $L_x$ is given as
  follows:
  \[
  L_x \cong
  \begin{cases}
    \emptyset,
    & (x \in \mathring{X} = X_{l+m+1} \setminus X_{l+n-1}), \\[4pt]
    T^1 \sqcup \cdots \sqcup T^1,
    & (x \in X_{l+n-1}).
  \end{cases}
  \]
\end{lem}

\begin{proof}
Recall Definition~\ref{def of T-pseudomanifold}~(b).
For $0 \le i \le n=1$, if
$
x \in X_{l+2i+(m-n)} \setminus X_{l+2(i-1)+(m-n)},
$
then the link $L_x$ is a $(2n-2i-1)$-dimensional locally standard
$T^{n-i}$-pseudomanifold.
When $2n-2i-1=-1$, we interpret $L_x=\emptyset$.
If $i=1$, then $L_x=\emptyset$.
If $i=0$, then $L_x$ is a $1$-dimensional locally standard
$T^1$-pseudomanifold.
By Proposition~\ref{1-dim}, each connected component of $L_x$
is homeomorphic to $T^1$.
Therefore,
$
L_x \cong T^1 \sqcup \cdots \sqcup T^1.
$
\end{proof}

\begin{rem}
By the above lemma, the link of each point is either empty or a disjoint union of circles.
In particular, no nontrivial finite quotient singularity appears.
Hence locally standard $T$-pseudomanifolds with $\dim X \leq 3$
have no orbifold singularities.
\end{rem}

The following lemma characterizes when a low-dimensional locally standard
$T$-pseudomanifold is a topological manifold in terms of connectedness of links.

\begin{lem}\label{mfd iff normal}
  A locally standard $T$-pseudomanifold $X$ with $\dim X \leq 3$ is a topological manifold
  if and only if $X$ is normal, i.e., the link of each point is connected.
\end{lem}

\begin{proof}
Since $X$ is a locally standard $T$-pseudomanifold with $\dim X \leq 3$,
we have $n=0$ or $1$.
If $n=0$, then $X$ is a topological manifold and every link is empty.
Hence the assertion holds.
Assume $n=1$.
Let $x \in X_{l+n-1}$.
By Definition~\ref{def of T-pseudomanifold},
there exists a $T$-invariant open neighborhood of $x$ of the form
\[
U_x
\cong
\mathbb{R}^l \times U(1)^{m-1} \times \cone(L_x).
\]
By Lemma~\ref{link is simple},
$
L_x \cong T^1 \sqcup \cdots \sqcup T^1.
$
If $L_x$ is connected, then
\[
\cone(L_x) \cong \cone(T^1) \cong \mathbb{R}^2.
\]
Since $U(1)^{m-1}$ is a manifold,
it admits a neighborhood homeomorphic to $\mathbb{R}^{m-1}$.
Hence $x$ admits an open neighborhood homeomorphic to
$\mathbb{R}^{l+m+1}$.

Conversely, assume that $L_x$ is disconnected.
Then $L_x$ is a disjoint union of several copies of $T^1$.
Hence $\cone(L_x)$ has a cone singularity,
and the cone vertex does not admit a neighborhood homeomorphic to
$\mathbb{R}^2$.
Taking the contrapositive, if $X$ is a topological manifold,
then the link of every point is connected.
\end{proof}

\begin{rem}\label{only one top stratum}
  By \cite[Lemma 2.6.3]{Fri20}, a connected normal topological stratified
  pseudomanifold has only one top stratum. Hence $Q$ also has only one
  top stratum.
\end{rem}

We next compare the normality of a locally standard
$T$-pseudomanifold and that of its orbit space.

\begin{prop}\label{normality of orbit space}
Let $\pi \colon X \to Q$ be the orbit projection of a locally standard
$T$-pseudomanifold.
Then $X$ is normal if and only if $Q$ is normal.
\end{prop}

\begin{proof}
Let $x \in X$, and let $L_x$ be the link of $x$.
By \cite[Section 3]{Pop00} and \cite[Proposition 4.1]{KK26},
the link of $\pi(x) \in Q$ is given by
\[
L_{\pi(x)} \cong L_x / T_x,
\]
where $T_x$ is the isotropy subgroup of $x$.
Since the action of the compact torus $T_x$ preserves connected components,
the quotient $L_x/T_x$ is connected if and only if $L_x$ is connected.
Hence $L_{\pi(x)}$ is connected if and only if $L_x$ is connected.
\end{proof}

We now combine Lemma~\ref{mfd iff normal} and
Proposition~\ref{normality of orbit space} to characterize when
$X$ is a topological manifold in terms of the orbit space $Q$.

\begin{cor}\label{mfd iff orbit normal}
Let $\pi \colon X \to Q$ be the orbit projection of a locally standard
$T$-pseudomanifold with $\dim X \leq 3$.
Then $X$ is a topological manifold if and only if $Q$ is normal.
\end{cor}

\subsection{The graph case}\label{sec 4.2}
Let $X$ be a locally standard $T$-pseudomanifold over a graph $Q$.
By Remark~\ref{only one top stratum}, if $X$ admits a structure of a topological manifold, then $Q$ must be either an interval or a single loop.
However, if $Q$ is a loop, then the link of the vertex of the loop consists of two points.
Hence $Q$ is not normal.
Therefore, by Corollary~\ref{mfd iff orbit normal}, $X$ is not a topological manifold.
Thus we obtain the following proposition.

\begin{prop}
  Let $X$ be a locally standard $T$-pseudomanifold over a graph $Q$. Then $X$
  is a topological manifold if and only if $Q$ is an interval.
\end{prop}

We next describe examples of locally standard $T$-pseudomanifolds over an interval.

\begin{ex}
Let $X$ be a locally standard $T$-pseudomanifold over a graph $Q$.

\begin{itemize}
\item
Suppose that $\dim X =2$.
Then $\dim T=m=1$.
In this case, $X$ is a locally standard $T$-pseudomanifold with semi-trivial filtration, and the characteristic data $(Q,\{1,T\},0)$ is given as follows:
\[
\begin{tikzpicture}

\draw (1,2) -- (4,2);

\node[above left] at (1,2) {$T^1$};
\node[above right] at (4,2) {$T^1$};

\fill (1,2) circle (2pt);
\fill (4,2) circle (2pt);

\end{tikzpicture}
\]
This coincides with the characteristic function introduced in \cite{DJ91}.
Hence $X$ is the quasitoric manifold $\mathbb{C}P^1$.
For quasitoric manifolds, see \cite{DJ91} and \cite{BP12}.

\item
Suppose that $\dim X =3$.
Then $\dim T=m=2$.
In this case, the characteristic functor assigns $\{1\}$ to the edge of the interval, and assigns a $1$-dimensional subtorus of $T^2$ to each vertex.
Giving a $1$-dimensional subtorus of $T^2$ is equivalent to assigning a primitive vector in $\mathbb{Z}^2$ up to sign.
Hence the characteristic data $(Q,\lambda,0)$ is given as follows:
\[
\begin{tikzpicture}

\draw (1,2) -- (4,2);

\node[above left] at (1,2) {$(a,b)$};
\node[above right] at (4,2) {$(c,d)$};

\fill (1,2) circle (2pt);
\fill (4,2) circle (2pt);

\end{tikzpicture}
\]
where $(a,b)$ and $(c,d)$ are primitive vectors.
This coincides with the hypercharacteristic function discussed in \cite{SS16}.
Therefore, by \cite[Proposition 2.3]{SS16}, $X$ is homeomorphic to one of the following spaces:
\[
X \cong
\begin{cases}
S^3,
& ( \det A=\pm 1), \\[4pt]
\mathbb{C}P^1 \times S^1,
& ( (c,d)=\pm (a,b)), \\[4pt]
S^3/(\mathbb{Z}/k\mathbb{Z}),
& ( \det A=\pm k).
\end{cases}
\]
where $A$ is the matrix whose rows are $(a,b)$ and $(c,d)$.  
Moreover, $S^3$ is the moment-angle manifold over an interval (see \cite{BP12}).
\end{itemize}
\end{ex}

\subsection{The case $(l, m, n)=(1, 1, 1)$}\label{sec 4.3}

Let $X$ be a $3$-dimensional locally standard $T$-pseudomanifold over a $2$-stratifold, which corresponds to the case $(l, m, n)=(1, 1, 1)$. 
If $X$ is a topological manifold, we obtain the following property for the $2$-stratifold $Q$ by Proposition~\ref{normality of orbit space}.

\begin{prop}\label{prop of mfd with boundary}
  A normal $2$-stratifold is a surface with boundary. Hence, it is obtained
  from a closed surface by removing finitely many open discs.
\end{prop}
\begin{proof}
We verify the local structure around each point of the $2$-stratifold
$(Q, \mathfrak{Q} \colon Q_2 \supset Q_1 \supset \emptyset)$.
If $p \in Q_2 \setminus Q_1$, then a small open neighborhood
$U_p$ of $p$ is homeomorphic to $\mathbb{R}^2$.
Suppose that $p \in Q_1$.
By Definition~\ref{def of pseudomanifold}-3,
a small open neighborhood $U_p$ of $p$ is homeomorphic to
\[
\mathbb{R}^1 \times \cone(L_p),
\]
where $L_p$ is the link of $p$.
Since $L_p$ is a $0$-dimensional compact topological stratified
pseudomanifold, it consists of finitely many points.
Since $Q$ is normal, the link $L_p$ is connected.
Therefore, $L_p$ consists of a single point.
Hence
\[
U_p
\cong
\mathbb{R}^1 \times \cone(L_p)
=
\mathbb{R}^1 \times \cone(\ast)
\cong
\mathbb{R}^1 \times \mathbb{R}_{\ge 0}.
\]
Thus every point of $Q$ admits a neighborhood homeomorphic either to
$\mathbb{R}^2$ or to $\mathbb{R} \times \mathbb{R}_{\ge 0}$.
Therefore, $Q$ is a manifold with boundary.
The second assertion follows from the classification of connected compact
surfaces with boundary (see, for example,
\cite[(4.17) Theorem]{Kin93}).
\end{proof}

The following proposition, which also follows from \cite[Theorem 1]{Ray68},
is a consequence of Corollary~\ref{class of over 2-stratifold}.

\begin{prop}
  Let $X$ be a connected $3$-dimensional locally standard $T$-pseudomanifold
  over a $2$-stratifold. If $X$ is a topological manifold, then its equivariant
  homeomorphism type is classified by the homeomorphism type of a surface with
  boundary.
\end{prop}

\begin{proof}
  By Corollary~\ref{class of over 2-stratifold}, the equivariant homeomorphism
  type is classified by the characteristic data $(Q, \{1,T\}, 0)$, and hence
  by the stratified homeomorphism type of the $2$-stratifold $Q$ alone (see Definition~\ref{def of isomorphism of char pairs} and Remark~\ref{rem free chern}). By
  Proposition~\ref{prop of mfd with boundary}, $Q$ is a surface with boundary.
  Since a homeomorphism of surfaces with boundary maps the interior to the interior and
  the boundary to the boundary, it is a stratified homeomorphism.
\end{proof}

\begin{rem}
  As part of the results of \cite[Theorem 1]{Ray68}, it is a classical
  theorem that $S^1$-actions on $3$-manifolds without exceptional orbits can be
  classified by surfaces with boundary. Thus, Corollary~\ref{class of over
  2-stratifold} can be seen as an extension of this classification result from
  manifolds to the framework of topological stratified pseudomanifolds.
\end{rem}

\subsection{Classification table}\label{sec 4.4}

We conclude this section by summarizing the conditions under which a connected
locally standard $T$-pseudomanifold with $\dim X \leq 3$ admits the structure
of a topological manifold.
The results are listed in Table~\ref{tab:when_manifold}.
\begin{table}[htbp]
\centering
\[
\begin{array}{c|cccc|c}
\dim X & \dim Q & l & m & n & \text{orbit space}\\
\hline
1 & 0 & 0 & 1 & 0 & \{\ast\} \\
\hline
2 & 0 & 0 & 2 & 0 & \{\ast\} \\
2 & 1 & 1 & 1 & 0 & S^1\\
2 & 1 & 0 & 1 & 1 & \text{interval}\\
\hline
3 & 0 & 0 & 3 & 0 & \{\ast\}\\
3 & 1 & 1 & 2 & 0 & S^1\\
3 & 1 & 0 & 2 & 1 & \text{interval}\\
3 & 2 & 2 & 1 & 0 & \text{closed surface}\\
3 & 2 & 1 & 1 & 1 & \text{compact surface with boundary}\\
\end{array}
\]
\caption{Orbit spaces for locally standard $T$-pseudomanifolds with $\dim X \leq 3$ that are topological manifolds.}
\label{tab:when_manifold}
\end{table}

\appendix

\section{Lemmas on topological stratified pseudomanifolds}\label{sec A}

In this appendix, we prepare several lemmas on topological stratified
pseudomanifolds.
Since all of them are basic properties, the reader may skip this appendix and
refer back to it only when necessary.
We begin with the following well-known property, called the Frontier
Condition.

\begin{prop}[{\cite[Lemma 2.3.7]{Fri20}}]\label{frontier condition}
  A topological stratified pseudomanifold satisfies that for any two strata
  $S$ and $S'$, if $S \cap \overline{S'} \neq \emptyset$, then
  $S \subset \overline{S'}$.
\end{prop}

By this property, the set $\mathcal{S}(Q)$ of strata of a topological
stratified pseudomanifold $Q$ carries the structure of a poset. The following
lemma makes this explicit.

\begin{lem}[{\cite[Proposition 2.2.20]{Fri20}}]\label{order from frontier condition}
  Let $Q$ be a topological stratified pseudomanifold. The closure of any
  stratum can be written as a union of strata of lower dimension. More
  precisely, for any stratum $S \in \mathcal{S}(Q)$, the following equality
  holds:
  \[
  \overline{S} = \bigcup_{\substack{R \in \mathcal{S}(Q), \\ R \subset \overline{S}}} R.
  \]
\end{lem}

We also recall the following well-known fact about topological manifolds.

\begin{lem}[{\cite[Exercise 1.2.6]{Hir76}}]\label{classification of 1-manifold}
  A connected, paracompact Hausdorff $1$-manifold is diffeomorphic to $S^1$ if it is compact, and to $\mathbb{R}$ if it is not compact.
\end{lem}

The next lemma summarizes several basic properties of top strata.

\begin{lem}\label{properties of top stratum}
Let $Q$ be a connected, compact topological stratified pseudomanifold with
filtration
\[
Q = Q_N \supsetneq Q_{N-1} \supset \cdots \supsetneq \emptyset.
\]
Then the following statements hold.

\begin{enumerate}
\item[(1)]
For any stratum $R \subset Q_k \setminus Q_{k-1}$ with $k < N$, there exists
a top stratum $S \subset Q_N \setminus Q_{N-1}$ such that
$
R \subset \overline{S}.
$

\item[(2)]
For any top stratum $S \subset Q_N \setminus Q_{N-1}$,

\begin{enumerate}
\item[(a)]
$S$ is open in $Q$;

\item[(b)]
$S$ is non-compact, equivalently, not closed in $Q$;

\item[(c)]
each connected component of $\overline{S} \setminus S$ contains a stratum
of lower dimension than $S$.
\end{enumerate}
\end{enumerate}
\end{lem}

\begin{proof}
We first prove (1).
  By Definition~\ref{def of pseudomanifold}, $Q_N \setminus Q_{N-1}$ is dense
  in $Q$. Since $R \subset Q$, we have $R \subset \overline{Q_N \setminus
  Q_{N-1}}$. Since $Q_N \setminus Q_{N-1}$ is the union of all top strata,
  there exists a top stratum $S \subset Q_N \setminus Q_{N-1}$ such that
  $R \cap \overline{S} \neq \emptyset$. Therefore, by the Frontier Condition
  (Lemma~\ref{frontier condition}), it follows that $R \subset \overline{S}$.

We next prove (2)(a).
Since $Q_{N-1}$ is closed in $Q$,
the complement
$
Q_N \setminus Q_{N-1}
$
is open in $Q$.
As $S$ is a connected component of
$
Q_N \setminus Q_{N-1},
$
it follows that $S$ is open in $Q$.
  For (2)(b), suppose for contradiction that $S$ is compact. Since $Q$ is a
  compact Hausdorff space, $S$ is closed. Combined with (2)(a), this means $S$ is
  both open and closed in $Q$, contradicting the connectedness of $Q$.
We finally prove (2)(c).
Since $S$ is not closed,
$
\overline{S} \setminus S \neq \emptyset.
$
Let $C$ be a connected component of
$
\overline{S} \setminus S,
$
and take a point $x \in C$.
Let $R \subset Q$ be the stratum containing $x$.
Since
$
R \cap \overline{S} \supset \{x\} \neq \emptyset,
$
the Frontier Condition implies
$
R \subset \overline{S}.
$
Moreover, since $x \notin S$,
we have $R \neq S$,
and hence $R$ is lower-dimensional than $S$.
By Lemma~\ref{order from frontier condition},
$
R \subset \overline{S} \setminus S.
$
Since $R$ is connected and intersects $C$,
we conclude that
$
R \subset C.
$
\end{proof}

Using the above properties of strata, we now identify the structure of a
compact topological stratified pseudomanifold with $l=0$ and $n=1$ as a finite graph.

\begin{lem}\label{structure of graph}
  Let $Q$ be a connected, compact topological stratified pseudomanifold with
  $l = 0$ and $n = 1$. Then $Q$ has the structure of a finite graph (possibly
  with loops and multiple edges).
\end{lem}

\begin{proof}
  Since $Q$ is compact, the number of strata is finite.
  By Lemma~\ref{properties of top stratum}-(2)(b) and
  Definition~\ref{def of pseudomanifold}-1,
  each $1$-stratum is a non-compact $1$-dimensional topological manifold.
  Hence, by Lemma~\ref{classification of 1-manifold},
  each $1$-stratum is homeomorphic to $\mathbb{R}$.
  It remains to show that each connected component of
  $\overline{S} \setminus S$
  for a $1$-stratum $S$
  contains a $0$-stratum,
  and that every $0$-stratum is contained in the closure of some
  $1$-stratum.
  The former follows from
  Lemma~\ref{properties of top stratum}-(2)(c),
  and the latter from
  Lemma~\ref{properties of top stratum}-(1).
\end{proof}

\end{document}